\documentclass[]{imsart}

\RequirePackage{amsthm,amsmath,amsfonts,amssymb}
\RequirePackage[authoryear]{natbib}

\RequirePackage[colorlinks,citecolor=blue,urlcolor=blue]{hyperref}
\RequirePackage{graphicx}
\usepackage{mathtools}

\startlocaldefs

\theoremstyle{plain}

\newtheorem{thm}{Theorem}[section]

\newtheorem{prop}[thm]{Proposition}

\theoremstyle{remark}
\newtheorem{rem}{Remark}
\newtheorem{definition}{Definition}
\newtheorem{assumption}{Assumption}

\endlocaldefs

\begin{document}

\begin{frontmatter}

\title{Generalized Bernoulli Process and Fractional Binomial Distribution}
\runtitle{Generalized Bernoulli Process and Fractional Binomial Distribution}

\begin{aug}
\author[]{\fnms{Jeonghwa}~\snm{Lee}\ead[label=a]{leejb@uncw.edu}}
\address[]{Department of Mathematics and Statistics, University of North Carolina Wilmington, USA\printead[presep={,\ }]{a}}

\end{aug}

\begin{abstract}
Recently, a generalized Bernoulli process (GBP) was developed as a stationary binary sequence whose covariance function obeys a power law. In this paper, we further develop generalized Bernoulli processes, reveal their asymptotic behaviors, and find applications. 
We show that a GBP can have the same scaling limit as the fractional Poisson process.  
Considering that the Poisson process approximates the Bernoulli process under certain conditions, the connection we found between GBP and the fractional Poisson process is thought of as its counterpart under long-range dependence. When applied to indicator data, a GBP outperforms a Markov chain in the presence of long-range dependence.
Fractional binomial models are defined as the sum in GBPs, and it is shown that when applied to count data with excess zeros, a fractional binomial model outperforms zero-inflated models that are used extensively for such data in current literature.
\end{abstract}


\end{frontmatter}
\section{Introduction}
\label{sec:intro}

The binomial and Poisson models are used for counting a number of events. Because of their common key assumptions on the memoryless property and constant rate of events, the Poisson distribution can approximate the binomial distribution. In practice, the assumptions often fail, resulting in overdispersion/excess zeros in count data.
There have been numerous models developed in the past for such data, among them are
zero-inflated models and generalized linear-type models
\cite{Ske,Alt,Kad,Wed,Con,Lam,Con2}.  

Some models were developed through a  correlated structure between events. 
In \cite{Rod}, the Markov-correlated Poisson process and Markov-dependent binomial distribution were developed through Markov-dependent Bernoulli trials.
In \cite{Pbo}, a new class of correlated Poisson process and correlated weighted Poisson process was proposed with equicorrelated uniform random variables.

In \cite{Lee}, a generalized Bernoulli process (GBP) was defined as a stationary binary sequence whose  covariance function  decreases slowly with a power law. This is called long-range dependence (LRD), also called long-memory property, and a stochastic process with LRD shows different large-scale behavior than
 the ``usual" stationary process whose covariance function decays exponentially fast.  LRD has been observed in many fields such as hydrology, econometrics, earth science, etc, and there have been several approaches to define LRD.  For more information on LRD, see \cite{Sam}. The estimation method for parameters in GBP and its application  to earthquake data can be found in \cite{Lee2},

In 2003, Nick Laskin discovered the fractional Poisson process by the fractional Kolmogorov-Feller equation that governs the non-Markov dependence structure. In the fractional Poisson process, events no longer independently occur, and  inter-arrival time follows the first  family of Mittag-Leffler distribution which is heavy-tail distribution  \cite{Las}. Long-memory property of the fractional Poisson process was investigated in \cite{Bia}.

In this paper, we further develop generalized Bernoulli processes, investigate their asymptotic properties, and find applications. 
 It turns out that there is an interesting connection between a GBP and the fractional Poisson process. Both can have the same scaling limit to the second family of Mittag-Leffler distribution, and in both GBP and fractional Poisson process, the interarrival time follows a heavy-tail distribution. Therefore, the large-scale behavior of the GBP is similar to that of the fractional Poisson process, and the GBP is considered to be a discrete-time counterpart of the fractional Poisson process. 
 
In the application of GBPs to indicator data in economics, it is shown that a GBP outperforms a Markov-dependent Bernoulli process when LRD is present in the data. Fractional binomial models are defined as the sum of the first $n$ variables in the GBPs, and it is shown that when applied to count data with over-dispersion/excess zeros, a fractional binomial model outperforms zero-inflated models.

In Section 2, we will review the GBP proposed in \cite{Lee}, and define a new generalized Bernoulli process. In Section 3, we will compare the GBPs and the fractional Poisson process by examining their asymptotic properties. In Section 4, we will continue to compare these distributions through simulations. Section 5 shows the applications of the GBPs and fractional binomial models to real data. All the proofs and technical results can be found in Supplementary material.  
Throughout this paper, we assume that $i, i_0,i_2,\cdots \in \mathbb{N},  i', i_0',i_2',\cdots \in \mathbb{N} ,$  $i_0<i_1<i_2<\cdots,$ and $i_0'<i_1'<i_2'<\cdots,$ unless mentioned otherwise. For any set $A,$ $|A|$ is the number of elements in $A$, with $|\emptyset|=0.$ $a_n\sim b_n$ means $a_n/b_n \to 1$ as $n\to \infty.$ For notational convenience, we denote $P((\cap_i A_i)\cap ( \cap_j B_j))$ by $P(\cap_i A_i \cap_j B_j).$

\section{Generalized Bernoulli processes and fractional binomial distributions}\label{sec2}

\subsection{ Generalized Bernoulli process I and fractional binomial distribution I}
We first review the  generalized Bernoulli process $\{X_i, i\in \mathbb{N}\}$  developed in \cite{Lee}. We will call it the generalized Bernoulli process I (GBP-I) to differentiate it from what we will develop in this paper. The GBP-I, $\{X_i, i\in \mathbb{N}\}$, was defined  with parameters, $p,H,c$ that satisfy the following assumption.
\begin{assumption}

 $p,H \in (0,1),$ and \[
    0\leq c<\min\{1-p, \frac{1}{2} ( -2p+2^{2H-2} +\sqrt{4p-p2^{2H}+2^{4H-4}})\}.
\] \label{Assumption 2.1}
\end{assumption}
Under Assumption \ref{Assumption 2.1}, the GBP-I is well defined with the following probabilities.
\begin{align*}
 &P(X_i=1)=p, P(X_i=0)=1-p ,\end{align*} 
\begin{equation*}
       P(X_{i_0}=1, X_{i_1}=1, \cdots, X_{i_n}=1)=
p\prod_{j=1}^{n}(p+c|i_j-i_{j-1}|^{2H-2}), \end{equation*} 
and for any disjoint sets $A,B \subset \mathbb{N},$ 
\begin{align}
    &P( \cap_{i'\in B }\{X_{i'}=0\}  \cap_{i\in A }   \{X_{i}=1\} )=\sum_{k=0}^{|B|}\sum_{\substack{B'\subset B\\ |B'|=k}} (-1)^{k}P(\cap_{i\in B'\cup A}\{X_{i}=1\} ),
\end{align}
and
\begin{align}
    P(\cap_{i'\in B }\{X_{i'}=0\})&=1+\sum_{k=1}^{|B|}\sum_{\substack{B'\subset B\\ |B'|=k}} (-1)^{k}P(\cap_{i\in B'}\{X_{i}=1\} ).
\end{align}

The following operators were defined in  \cite{Lee} to express the above probabilities more conveniently. 

\begin{definition} 
Define the following operation on a set $ A=\{i_0, i_1, \cdots,i_n\}\subset N\cup \{0\}$ with $i_0< i_1< \cdots<i_n. $
\[  L_H(A)= \prod_{j=1,\cdots,n}\Big(p+c|i_j-i_{j-1}|^{2H-2}\Big)  .\]
If $A=\emptyset$, define $ L_H(A):=1/p$, and if $|A|=1, L_H(A):=1. $ 
\end{definition}

 \begin{definition}
 Define for disjoint sets, $A,B\subset \mathbb{N}\cup\{0\}$ with $ |B|=m>0,$  \begin{align*}
    D_H(A,B)=
    \sum_{j=0}^{|B|}\sum_{ \substack{B'\subset B\\ |B'|=j}}(-1)^{j} L_H(A\cup B').
 \end{align*} 
    If $B=\emptyset, D_H(A,B):=L_H(A).$
    
    \end{definition}

 Using the operators, the probabilities in the GBP-I can be expressed as
\begin{align*}
     P(  \cap_{i\in A}\{X_{i}=1\} \cap_{i'\in B}\{X_{i'}=0\})&= pD_H(A,B)
\end{align*}
for disjoint sets $A,B \subset \mathbb{N}.$
The GBP-I is a stationary process with covariance function
\[ Cov(X_i, X_j)=pc|i-j|^{2H-2}, i\neq j.\]

 When $H\in(.5,1),$  the GBP-I possesses LRD since $\sum_{i=1}^{\infty} Cov(X_1, X_i)=\infty.$
The sum of the first $n$ variables in the GBP-I was defined as the fractional binomial random variable whose mean is $np$, and if $H\in (.5,1),$ its variance is asymptotically proportional to $n^{2H}.$
We will call it fractional binomial distribution I, and denote it by $B_n(p,H,c)$ or simply $B_n$. When $c=0, $ $B_n(p,H,0)$ becomes the regular binomial random variable whose parameters are $n,p.$

\subsection{Generalized Bernoulli process II and fractional binomial distribution II  }
 We define the generalized Bernoulli process II (GBP-II),  $\{X_i^{(n)},i=1,2,\cdots, n\}$, $n \in \mathbb{N}$, that has three parameters $H,c,\lambda,$ and the fractional binomial random variable II, $B_n^{\circ}(H,c,\lambda)=\sum_{i=1}^n X_i^{(n)}.$ To ease our notation, $X_i$ and $X_{i}^{(n)}$ are used interchangeably, and $B_n^{\circ}$ will replace $B_n^{\circ}(H,c,\lambda)$ when there is no confusion.
For the GBP-II, it is assumed that the parameters satisfy the following condition.

\begin{assumption}
 $H \in(.5, 1), c \in (0,2^{2H-2}),$ and $ \lambda \in (0,c).$ \label{Assumption 2.2}
\end{assumption}

We will show that the GBP-II is well-defined under Assumption \ref{Assumption 2.2} with the following probabilities. For $n\in\mathbb{N},$  $\{X_i^{(n)},i=1,2,\cdots, n\}$ is defined with    $
    P(X_{i}=1)=p_n,  P(X_{i}=0)=1-p_n
    $ for $i=1,2,\cdots,n,$
 and for any  $1\leq  i_0<i_1 <i_2<\cdots<i_k 
 \leq n,$
\begin{equation}
   P(X_{i_0}=1, X_{i_1}=1 \cdots X_{i_k}=1)= p_nc^k|(i_1-i_0) \cdot (i_2-i_1) \cdots (i_k-i_{k-1})|^{2H-2}  \end{equation}  where
   $p_n=\lambda n^{2H-2}.$

\begin{definition}
Define the following operation on a set $ A=\{i_0, i_1, \cdots,i_k\}\subset \mathbb{N}\cup\{0\}, i_0<i_1<\cdots<i_k.$
\[  L_H^{\circ} (A)= \prod_{j=1,\cdots,k}|i_j-i_{j-1}|^{2H-2}  .\]
\end{definition}
If $|A|=1$, then we define $L_H^{\circ}(A):=1.$  If $|A|=0$ (when $A=\emptyset$), then   $L_H^{\circ}(A):=c/p_n.$
For example, if $A=\{1,2,4,7\},  L_H^{\circ}(A)=|(2-1)(4-2)(7-4)|^{2H-2}.$ \\
 \begin{definition}
 Define for disjoint sets $A,B$ such that $A,B\subset \mathbb{N}\cup\{0\}$ and $ B\neq \emptyset,$ \begin{align*}
    D_H^{\circ}(A,B) =\sum_{j=0}^{|B|}\sum_{ \substack{B'\subset B\\ |B'|=j}}(-1)^{j}c^{j}  L_H^{\circ}(A\cup B'). \end{align*}  If $B=\emptyset,$ then $D_H^{\circ}(A,B):=L_H^{\circ}(A).$
    \end{definition} 
  The joint probabilities in GBP-II is defined with (3) and (1-2) for any disjoint sets $A, B \subset \{1,2,\cdots,n\}$, i.e., by the inclusion-exclusion principle.
  This can be succinctly written as  \begin{align}
 P(\cap_{i\in A} \{X_{i}=1\}\cap_{i'\in B} \{X_{i'}=0\})&=p_nc^{|A|-1}D_H^{\circ}(A,B) .\end{align}

To show that the GBP-II is well-defined, we have to verify that $(4)\geq 0$ for any disjoint sets $A,B \subset \{1,2,\cdots,n\},$ for any $n\in \mathbb{N}.$

\begin{prop}
Under Assumption \ref{Assumption 2.2}, for any $n\in \mathbb{N}$ and any disjoint sets $A,B\subset \{1,2,\cdots, n\},$ 
\[ p_nc^{|A|-1}D_H^{\circ}(A,B)>0.\] \label{Prop}
\end{prop} 
From the result of Proposition \ref{Prop}, the GBP-II is well defined stationary  binary sequence whose correlation function obeys a power law asymptotically.

\begin{thm} For any $n\in \mathbb{N},$ $\{X_i^{(n)}, i=1,2,\cdots, n \}$ defined with (4) under Assumption \ref{Assumption 2.2} is a stationary process   with $P(X_{i}=1)=p_n,P(X_{i}=0)=1-p_n,$ and
\\{\it i)} \[Cov(X_i^{(n)}, X_j^{(n)})=p_nc|i-j|^{2H-2}-p_n^2,\]
for $i\neq j, i,j=1,2,\cdots,n.$
\\{\it ii)} For any $i,j \in \mathbb{N},$ 
\begin{align*}
\lim_{n\to \infty}Corr(X_i^{(n)}, X_j^{(n)})= c|i-j|^{2H-2}.\end{align*}
  \\{\it iii)} Define $B_n^{\circ}=\sum_{i=1}^n X_i,$ then    $E(B_n^{\circ})=\lambda n^{2H-1}$, and as $n\to \infty,$
  \begin{align*}
      E((B_n^{\circ})^2)&\sim \frac{\lambda c}{(2H-1)H} n^{4H-2},\\
      Var(B_n^{\circ})&\sim n^{4H-2}  \Bigg(\frac{\lambda c}{(2H-1)H}-\lambda^2\Bigg).
  \end{align*}
  \label{Theorem 2.4}
\end{thm}
We call the stationary process $\{X_i^{(n)}, i=1,2,\cdots, n \}$ defined in Theorem \ref{Theorem 2.4} the generalized Bernoulli process II (GBP-II). Also, the sum of the first $n$ variables in the GBP-II,  $B_n^{\circ}$, is called the fractional binomial random variable II.
\begin{rem}
If we use a correlation function and its asymptotic behavior to define long-range dependence (LRD), then we conclude that the GBP-II always possesses a long-memory property, since \[  \sum_{i=1}^n corr(X_1^{(n)}, X_i^{(n)})\sim (c/(2H-1)-\lambda) n^{2H-1} ,\]  therefore, \[  \lim_{n\to \infty}\sum_{i=1}^n corr(X_1^{(n)}, X_i^{(n)})=\infty. \]
Alternatively, one can define LRD in the GBP-II using the asymptotic behavior of its covariance function. Since $\sum_{i=1}^n Cov(X_1^{(n)}, X_i^{(n)}) \sim (c\lambda/(2H-1)-\lambda^2) n^{4H-3},$ if $H\in (3/4,1),$ then \[\lim_{n\to \infty} \sum_{i=1}^n Cov(X_1^{(n)}, X_i^{(n)}) =\infty ,\] and the GBP-II is considered to have LRD when $H\in (3/4,1)$. 

\end{rem}

\section{GBP and  fractional Poisson process}
\subsection{Comparison between  the GBP-II and the fractional Poisson process}
We will show the connection between GBP-II and  the fractional Poisson process
by using the moment generating function (mgf).
First, we modify the GBP-II, and define $\{X_i^*, i =1,2,\cdots\}$ such that
\begin{align*}
    P( X_i^{*}=1)=ci^{2H-2}=cL_H^{\circ}(\{0,i\}),\end{align*} 
 \begin{align*}
P( X_{i_1}^{*}=1,X_{i_2}^{*}=1\cdots X_{i_k}^{*}=1 )&= c^k|i_1(i_2-i_1)(i_3-i_2)\cdots(i_k-i_{k-1})|^{2H-2}\\&=c^kL_H^{\circ}(\{0, i_1,\cdots,i_k\}),
\end{align*}
and in general, for $A,B\subset \mathbb{N}, A\cap B=\emptyset, $  \begin{align}
 P(\cap_{i\in A} \{X_{i}^{*}=1\}\cap_{i'\in B} \{X_{i'}^{*}=0\})&=c^{|A|}D_H^{\circ}(A\cup\{0\},B). \end{align}
Note that $D_H^{\circ}(A\cup\{0\},B)=D_H^{\circ}(A^{(1)}\cup\{1\},B^{(1)})$ where $A^{(j)}=\{ i+j :i\in A\}$ and $ B^{(j)}=\{ i+j :i\in B\} .$
Therefore, by Proposition \ref{Prop}, $(5)>0$, and  $\{X_i^{*}, i =1,2,\cdots\}$ is well defined, which we will call GBP-II$^*$.

Roughly speaking, $\{X_i^{*}, i =1,2,\cdots\}$ can be considered as what is observed after the first 1 in the GBP-II for large $n$. In fact, the probability distribution (5) is the limiting distribution of the conditional probability of the GBP-II that is observed after the first observation of "1".
If we define a random variable $T_0$ as the time when the first "1" appears in the GBP-II, then for any $t,i\in \mathbb{N},$
\[\lim_{n\to \infty}{P(X_{i+t}^{(n)}=1| T_0=t)}=ci^{2H-2} ,\]
and for $A,B\subset \mathbb{N}, A\cap B=\emptyset, $  \begin{align*}
\lim_{n\to \infty} P(\cap_{i\in A^{(t)}} \{X_{i}^{(n)}=1\}\cap_{i'\in B^{(t)}} \{X_{i'}^{(n)}=0\}|T_0=t)&=c^{|A|}D_H^{\circ}(A^{(t)}\cup\{t\},B^{(t)})\\&=c^{|A|}D_H^{\circ}(A\cup\{0\},B), \end{align*}
 which is the same as the distribution of $\{X_i^*, i\in\mathbb{N}\}$ defined in (5).

Define $B_n^{\circ,{*}}(H,c) $ as the sum of the first $n$ variables in the GBP-II$^*$, 
\[   B_n^{\circ,{*}}= \sum_{i=1}^{n}X_i^{*} ,\] and call it the fractional binomial random variable  II$^*.$
 In both of the fractional binomial distributions-II, II$^*$,  the moments are asymptotically proportional to the powers of $n^{2H-1}.$

\begin{thm} For any $k\in \mathbb{N},$
as $n \to \infty,$ \\ {\it i)} the fractional binomial II has
\[ E(({B_n^{\circ}})^k)\sim c_k n^{(2H-1)k},\] where
$c_k=\frac{k!\lambda(c\Gamma(2H-1))^{k-1}}{\Gamma((2H-1)(k-1)+2)},$
\\ {\it ii)} the fractional binomial II$^*$ has
\[ E(({B_n^{\circ,{*}}})^k)\sim c_k^* n^{(2H-1)k},\] where
$c_k^*=\frac{k!(c\Gamma(2H-1))^{k}}{\Gamma((2H-1)k+1)}.$
\label{Theorem 3.1}
\end{thm}

It turns out that a scaled fractional binomial II$^*$ has the same limiting distribution as a scaled fractional Poisson distribution. 
More specifically, the scaled fractional binomial II$^*$ and scaled fractional Poisson, 
$B_n^{\circ,{*}}(H,c)/n^{2H-1} $ and $N_{2H-1, c\Gamma(2H-1)}(n)/n^{2H-1},$ converge in distribution to the  second family of Mittag-Leffler random variable of order $2H-1,$ as $n\to \infty.$  \begin{thm} For the fractional binomial distribution II$^*$ with parameters $H,c$ that satisfy Assumption 2, and the fractional Poisson process with parameters $\mu=2H-1, \nu=c\Gamma(2H-1), $  the following holds:  Both   $B_n^{\circ,{*}}/n^{\mu}$  and  $N_{\mu,\nu}(n)/n^{\mu}$  converge in distribution, as $n\to \infty,$ to the second family of Mittag-Leffler random variable of order $\mu$,   $X_{\mu}$,
whose mgf is Mittag-Leffler function, i.e., \[E(e^{tX_{\mu}})=E_{\mu}(\nu t) \text{ for } t\in \mathbb{R}, \]  
where
\[E_{\mu}(z)=\sum_{k=0}^\infty \frac{z^k}{\Gamma(\mu k+1)}.\]  
\label{Theorem 3.3}
\end{thm}

 \subsection{Comparison between the GBP-I and the GBP-II}
We will compare the asymptotic properties of GBPs through asymptotic moments and tail behavior of return time. 
 
First, we define GBP-I$^*$  in a similar way that we defined the GBP-II$^*$.
GBP-I$^*$,  $\{X_i^*,i=1,2,\cdots\}$, is what is observed after the first ``1" in the GBP-I. Then it has 
\[
P(X_i^*=1)=P(X_{i+t}=1|{T_0}=t)=p+ci^{2H-2},\]
    \begin{align*} P( \cap_{i\in A}X_i^*=1\cap_{i'\in B}X_{i'}^*=1  )&=P(\cap_{i\in A}X_{i+t}=1\cap_{i'\in B}X_{i'+t}=1|{T_0}=t)\\&=D_H(A\cup\{0\},B),
\end{align*} for $A,B\in \mathbb{N}, A\cap B=\emptyset,$
where $T_0$ is the first time when ``1" appeared in the GBP-I. We also define $ B_n^*(p,H,c)$ as the sum of the first $n$ variables in the GBP-I$^*$, and call it the fractional binomial random variable I$^*$. 

\begin{thm}
For the fractional binomial I, $B_n,$ and the fractional binomial I$^*,$ $B_n^*,$ we have the following asymptotic properties as $n\to \infty. $
 \\
{\it i)} For the fractional binomial I,  the central moments are
\begin{align*}
 E((B_n-np)^2)&\sim\begin{dcases}
    b_2^{(1)} n  &\text{ if $H\in(0, .5),$}\\
   b_2^{(2)} n\ln{n}  &\text{ if $H = .5,$} \\
    b_2^{(3)} n^{2H}  &\text{ if $H\in(.5, 1),$}
\end{dcases} \\
  E((B_n-np)^k)&\sim  
   b_k n^{2H-2+k}   \text{ for $ k=3,4,\cdots,$ and $H \in(0, 1),$ }  \end{align*}
where $ b_2^{(1)}=  p(1-p)+2pc/(1-2H), b_2^{(2)}=  2pc,  b_2^{(3)}= {pc  }/{(H(2H-1))}, b_k=  k(k-1)cp(-p)^{k-2}/((2H-2+k)(2H-3+k) ),$   and $\lceil k/2 \rceil$ is the largest integer smaller than or equal to $k/2$. \\
{\it ii)} For the fractional binomial I$^*,$
\begin{align*}
 E((B_n^*-np)^2)&\sim\begin{dcases}b_2^{(1)} n  &\text{ if $H\in(0, .5),$}\\
   b_2^{(2)} n\ln{n}  &\text{ if $H = .5,$} \\
    b_2^{(3,*)} n^{2H}  &\text{ if $H\in(.5, 1),$}
\end{dcases} \\
 E((B_n^*-np)^k)&\sim  
  b_k^* n^{2H-2+k}    \text{ for $ k=3,4,\cdots,$ and $H \in(0, 1),$}  \end{align*}
where 
$b_2^{(3,*)}= {pc  }/{(H(2H-1))}-pc/H $ and $b_k^*=kc(-p)^{k-1}/(2H-2+k ) +k(k-1)cp(-p)^{k-2}/((2H-2+k)(2H-3+k) )
.$
\\

For the non-central moments, $E(B_n)=np,$ $  E(B_n^*)\sim np,$ and  
  $E(B_n^k)\sim n^kp^k,$ 
  $E((B_n^*)^k)\sim n^kp^k
$ for $k\geq 2.$ 
 

\end{thm}

Let's consider the number of 0's between two successive 1's plus one as a return time, also called interarrival time, and denote the return time from the $(i-1)^{th}$ 1 to  the $i^{th}$ 1  in the GBP-I as $T_i.$
In the GBP-I, return times, $\{T_i,i=2,3,\cdots,\},$ are i.i.d. with
\begin{align*}
 P(T_i=k)&=\begin{cases}D_H(\{1,k+1\},\{2,3,\cdots,k\}) &\text{ for } k\in \mathbb{N}/\{1\},\\L_H(\{1,2\}) &\text{ for } k=1. \end{cases}\end{align*} We will drop the subscript and use $T$ as a random variable for return time in the GBP-I.

In the GBP-II  $\{X_i^{(n)}\}$, we define a return time in the same way and denote $ T_i^{(n)}$ as the return time to the $i^{th}$ 1. Note that 
$ T_i^{(n)}, i=2,3,\cdots,$ are independent but not identically distributed, since the length of the sequence  $\{X_i^{(n)},i=1,\cdots,n\}$ is fixed as $n$. However, they are asymptotically i.i.d. when $n\to \infty,$ with its limiting distribution  
\begin{align}
 \lim_{n\to \infty}P( T_i^{(n)}=k)=\begin{cases}cD_H^{\circ}(\{1,k+1\},\{2,3,\cdots,k\}) &\text{ for } k\in \mathbb{N}/\{1\},\\c L_H^{\circ}(\{1,2\}) &\text{ for } k=1, \end{cases}  \end{align}
which is in fact the distribution of return time  $T^{\circ}$ in the GBP-II$^*$, if we define return time in the GBP-II$^*$, $\{X_i^{*}, i \in \mathbb{N}\},$  in the same way.  In the GBP-II$^*$, denote $T_i^{\circ}$ as the return time to the $i^{th}$ 1, it is easily derived that return times $T_1^{\circ}, T_2^{\circ}, \cdots$ are i.i.d. with distribution (6).

It is found that the return time of the GBP-I has a finite mean,  whereas the return time of the GBP-II$^*$ has an infinite mean. Also, the return time follows heavy-tail distribution in both the GBP-I and the GBP-II$^*$ with
 \[P(\text{return time}>t) \sim c t^{-\alpha} \]
for a large $t$, some constant $c$, and $\alpha \in (1,3)$ for the GBP-I, and $\alpha \in (0,1) $ for the GBP-II.

\begin{thm}In the GBP-I,
$E(T)=1/p.$ 
If $H\in (0, .5), var(T)<\infty,$ and
if  $H\in [.5, 1),   var(T)=\infty.$ 
Furthermore,\[ P(T>t)=t^{2H-3} L_1(t), \] for $H \in (0,1),$ where $L_1$ is a slowly varying function that depends on the parameters $H, p, c$ of the GBP-I.\end{thm}

\begin{thm}
In the GBP-II,  $E(T^{\circ})=\infty, $ and
\[P(T^{\circ}>t  )= t^{1-2H} L_2(t),\] for
$H\in (0.5,1),$
where $L_2$ is a slowly varying function that depends on the parameters $H,\lambda, c$ of the GBP-II.
\end{thm}

\section{Simulation}
In this section, we examine and compare the shape of the fractional binomial distributions (FB) and the fractional Poisson distribution through simulations. Each histogram in Figures 1-4 was made of 3000 simulated random variables of the corresponding distribution.   Figure \ref{Figure 1} shows the histograms of the scaled FB-II$^*$, $B_n^{\circ,{*}}/n^{2H-1},$  and the scaled fractional Poisson random variable, $N_{\mu,\nu}(n)/n^{2H-1}, $ for various parameters $H,c, n,  \mu=2H-1,\nu=c\Gamma(2H-1).$ It is observed that when $n=50,$ the histograms of the scaled FB-II$^*$  and the scaled fractional Poisson random variable are fairly close to the pdf of the second family of Mittag-Leffler distribution. The approximation is better for larger $n.$  This result reflects well Theorem \ref{Theorem 3.3}.   

Figure \ref{Figure 2} shows the histograms of the FB-II, II$^*$ and the fractional Poisson distribution for various parameters. It is seen that the histograms of FB-II$^*$ and the fractional Poisson largely overlap, which is not surprising given the fact that their scaled distributions have the same limiting distribution. However, the FB-II behaves quite differently than the other two distributions as it has a high peak near 0.    

The FB-I, I$^*$ and the binomial distribution are compared in Figure \ref{Figure 3}. The binomial distribution is roughly symmetric and bell-shaped for each set of parameters as expected from the central limit theorem, whereas the FBs show various shapes and larger variability than the binomial distribution. Unlike the FB-I$^*$, the FB-I  has a large probability near 0, a similar phenomenon observed in  Figure \ref{Figure 2} for the FB-II.

Figure \ref{Figure 4} are the combined results of Figures \ref{Figure 2}, \ref{Figure 3}, putting together the histograms of the binomial, FB-I$^*$, II$^*$, and fractional Poisson distributions for each set of parameters. The FB- II$^*$ and the fractional Poisson distributions are similar to each other, and the shape and the range of these distributions are neither close to the binomial nor the FB- I$^*$.
\begin{figure}[htp]
    \centering
\includegraphics[width=.4\textwidth]{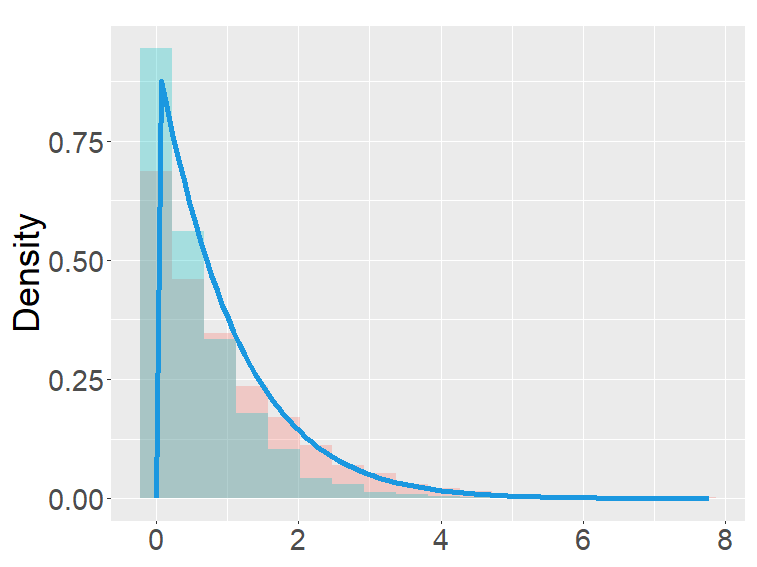}
\includegraphics[width=.4\textwidth]{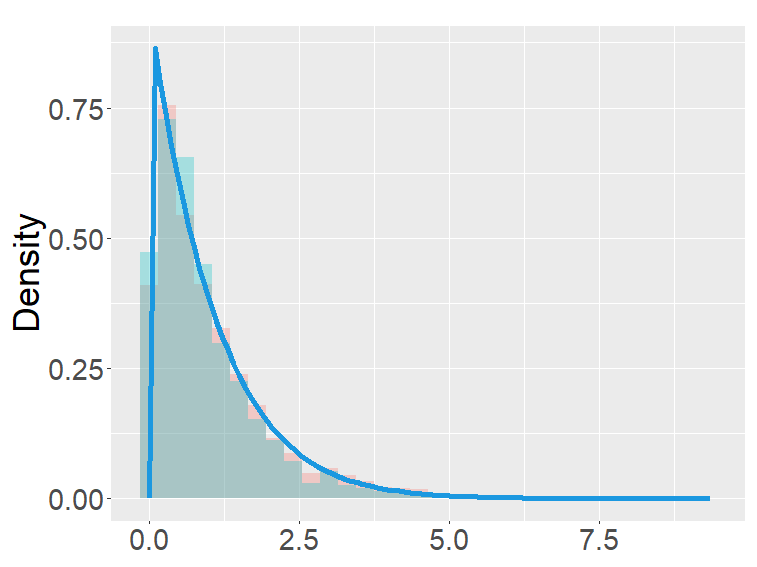}

\includegraphics[width=.4\textwidth]{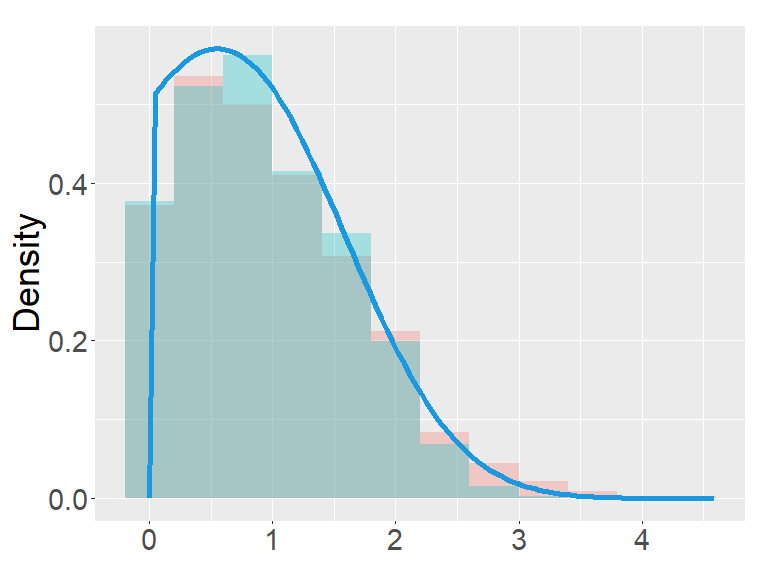}
\includegraphics[width=.4\textwidth]{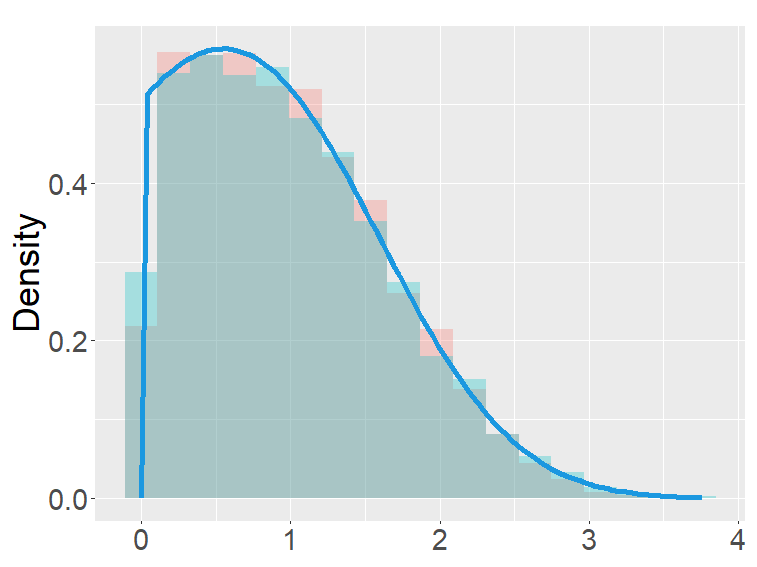}

\includegraphics[width=.4\textwidth]{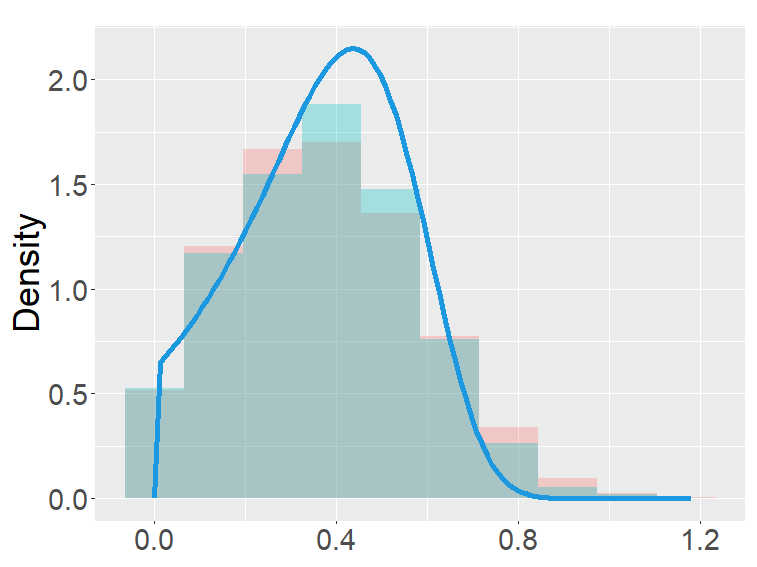}
\includegraphics[width=.4\textwidth]{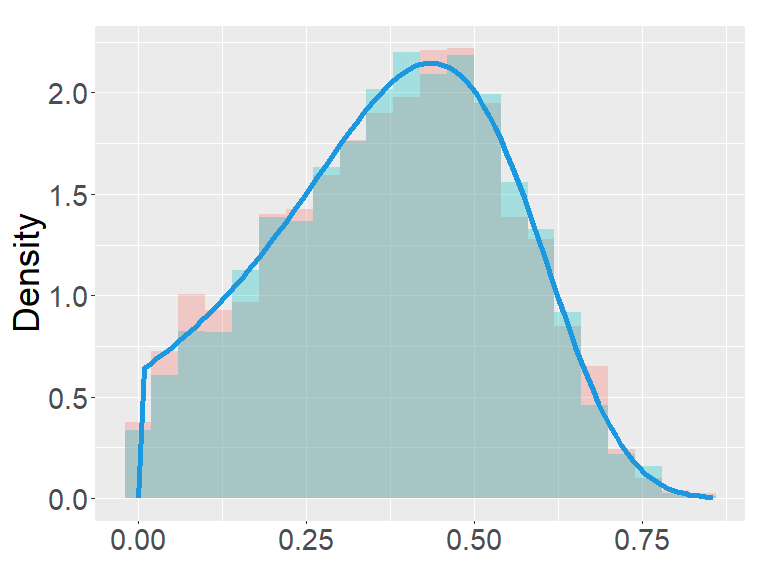} 

\includegraphics[width=1\textwidth]{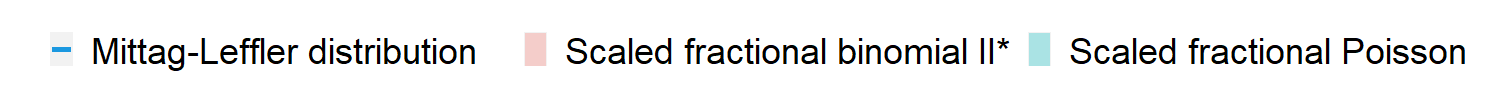}

\caption{The first row: $(H,c)=(.6,.2)$  with $n=50$ (left), 1000 (right),\\
the second row: $(H,c)=(.8,.6)$ with $n=50$ (left), 1000 (right),\\
the third row: $(H,c)=(.9,.3)$ with $n=50$ (left), 1000 (right). }
\label{Figure 1}
\end{figure}

\begin{figure}
    \centering
    {\includegraphics[width=0.4\textwidth]{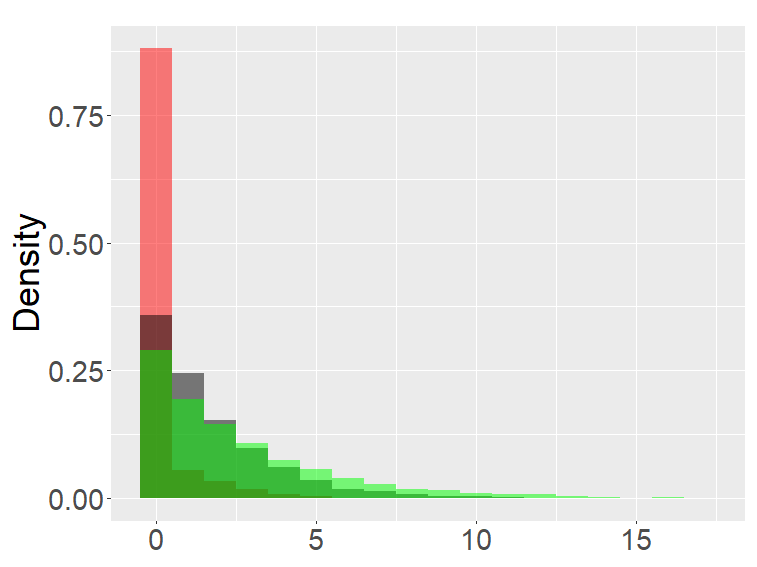}} 
   {\includegraphics[width=0.4\textwidth]{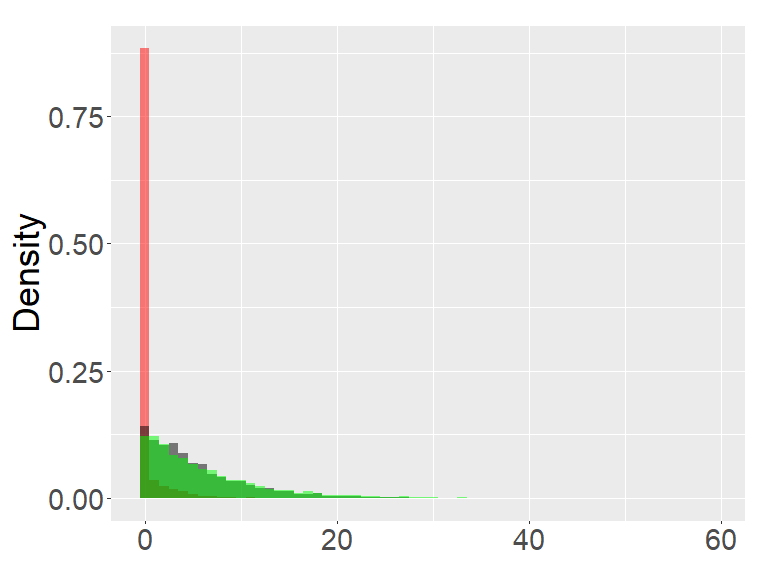}} 
{\includegraphics[width=0.4\textwidth]{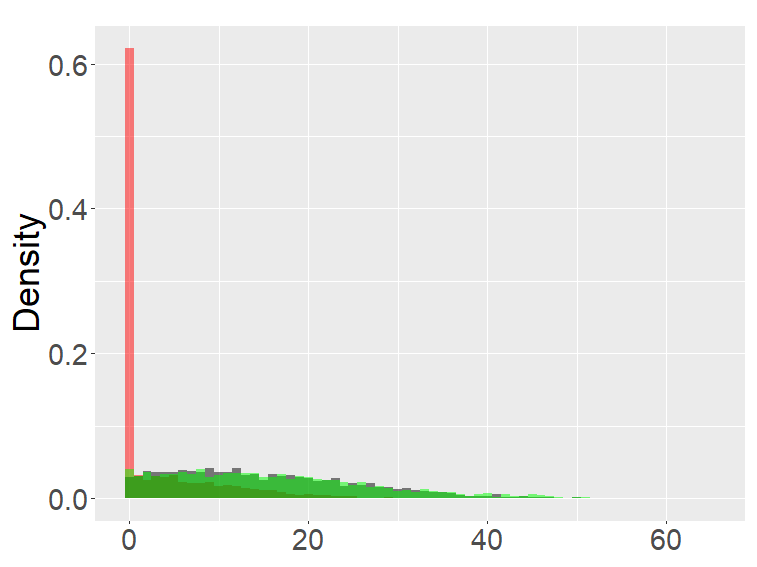}}
{\includegraphics[width=0.4\textwidth]{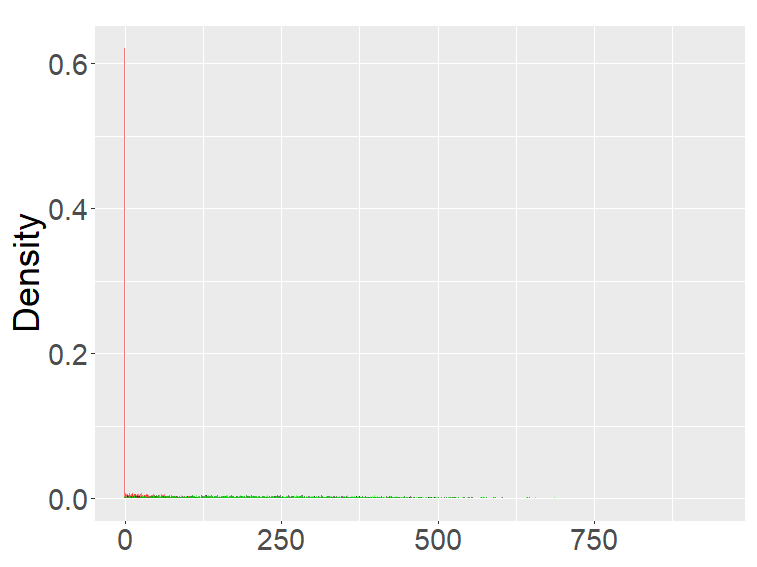}}


   {\includegraphics[width=0.4\textwidth]{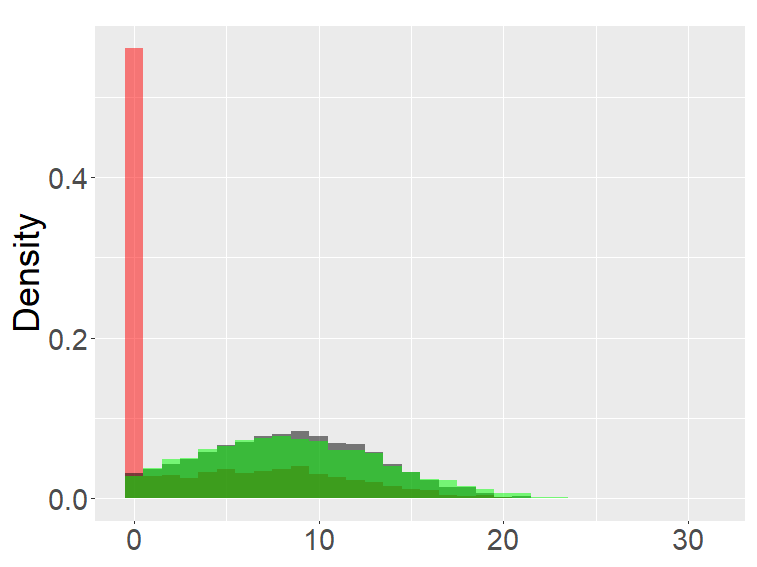}}
{\includegraphics[width=0.4\textwidth]{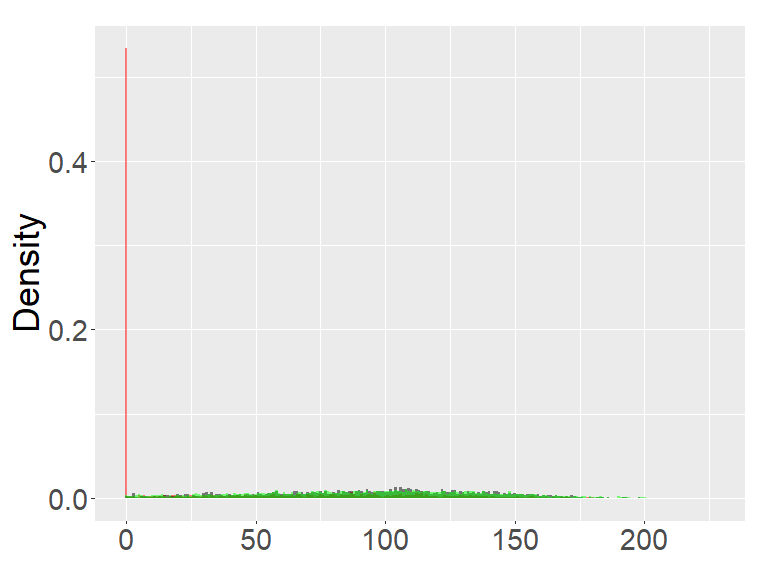}}
    \includegraphics[width=1\textwidth]{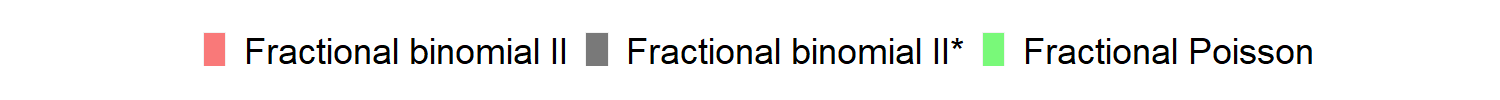}
  \caption{The first row: $(H,c)=(.6,.2)$  with $n=50$ (left), 1000 (right),\\
the second row: $(H,c)=(.8,.6)$ with $n=50$ (left), 1000 (right),\\
the third row: $(H,c)=(.9,.3)$ with $n=50$ (left), 1000 (right). 
($\lambda=c/2$ for all graphs.)}
\label{Figure 2}
\end{figure}

\begin{figure}
    \centering
   {\includegraphics[width=0.4\textwidth]{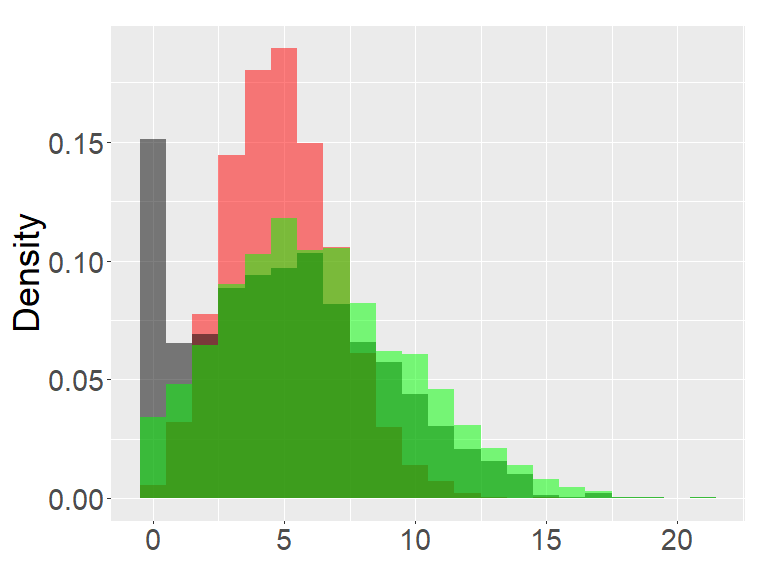}} 
    {\includegraphics[width=0.4\textwidth]{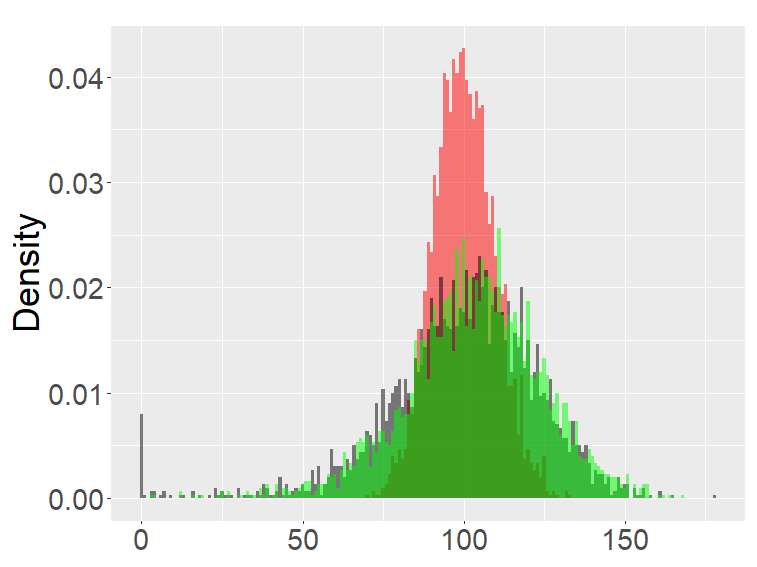}} 
{\includegraphics[width=0.4\textwidth]{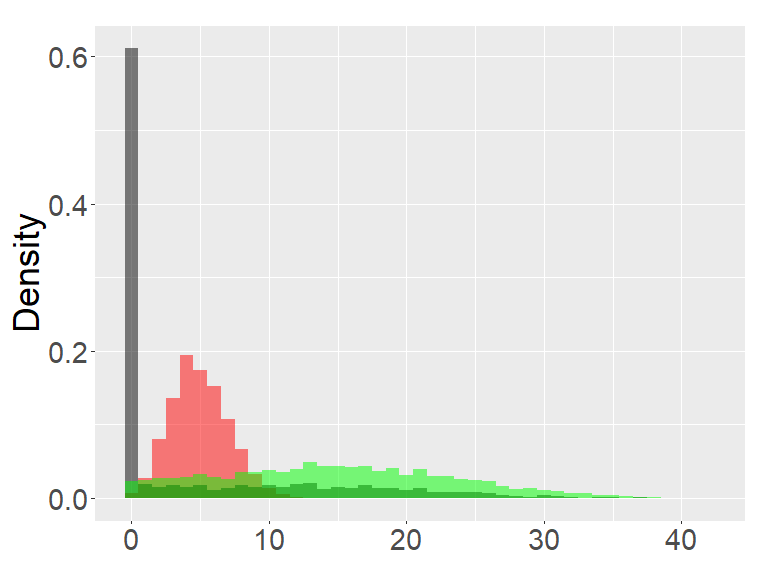}}
{\includegraphics[width=0.4\textwidth]{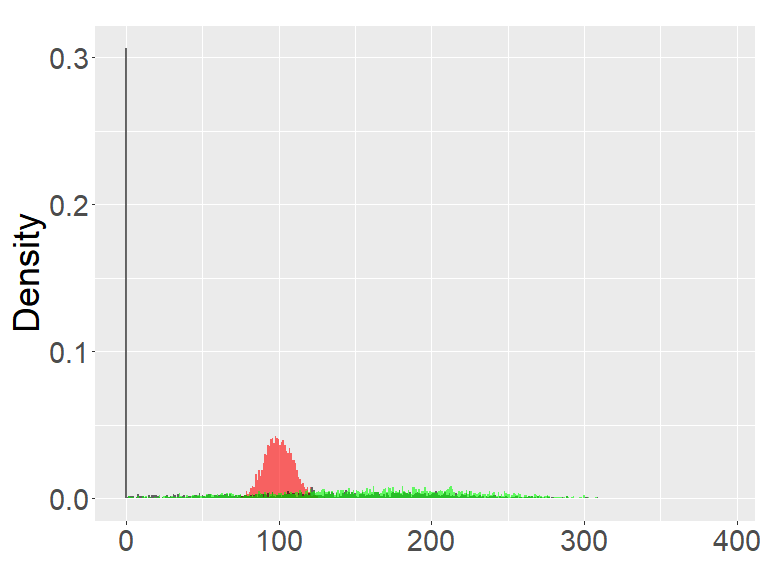}}


  {\includegraphics[width=0.4\textwidth]{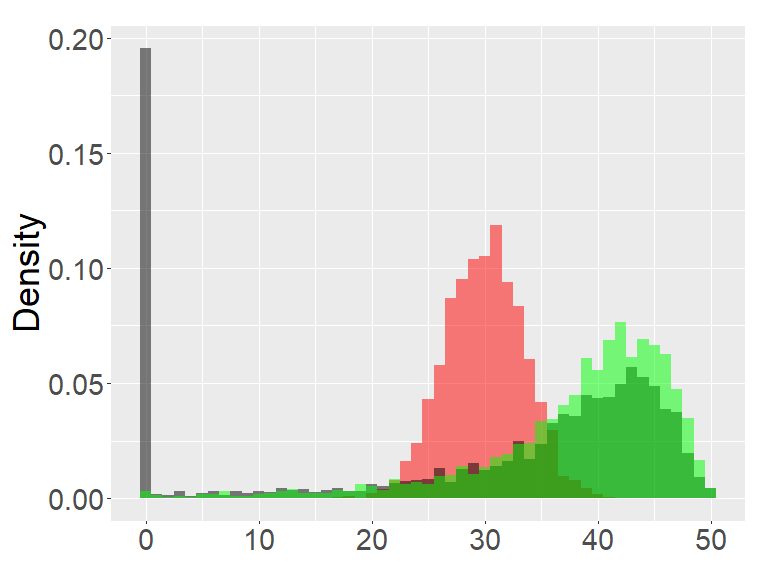}}
  {\includegraphics[width=0.4\textwidth]{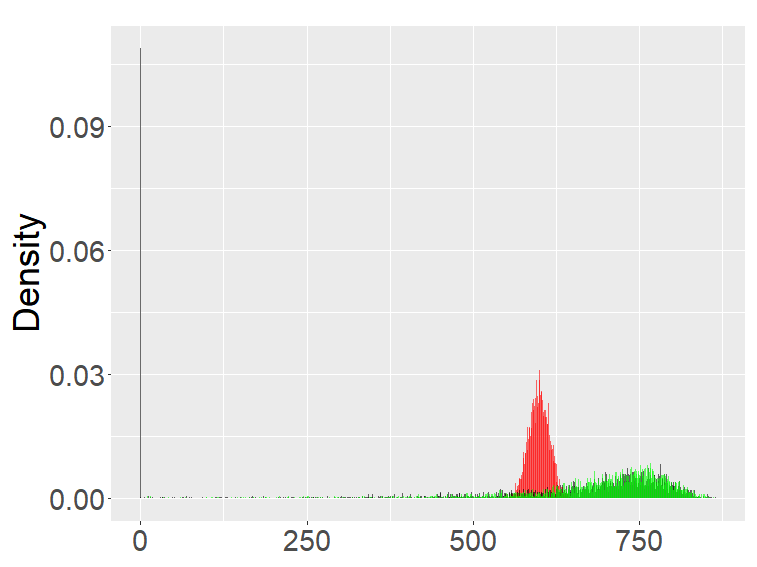}}
   \includegraphics[width=1\textwidth]{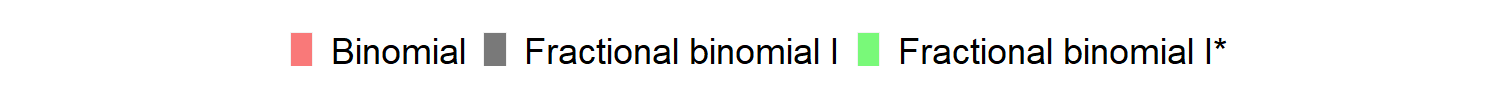}
  \caption{The first row: $(p,H,c)=(.1, .6, .2)$ with $n=50$ (left), 1000 (right),\\
the second row: $(p,H,c)=(.1, .8, .6)$ with $n=50$ (left), 1000 (right), \\
the third row: $(p,H,c)=(.6, .9, .3)$ with $n=50$ (left), 1000 (right).} \label{Figure 3}
\end{figure}

\begin{figure}
\centering
\includegraphics[width=.45\textwidth]{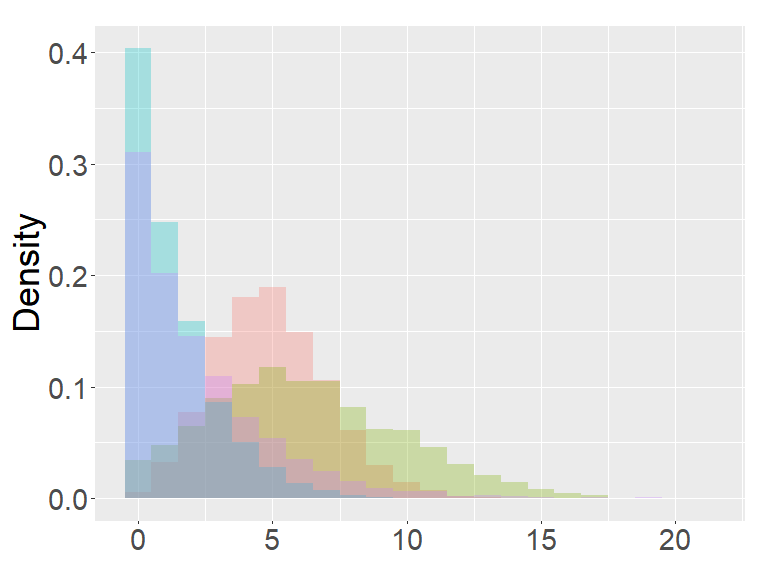}\includegraphics[width=.45\textwidth]{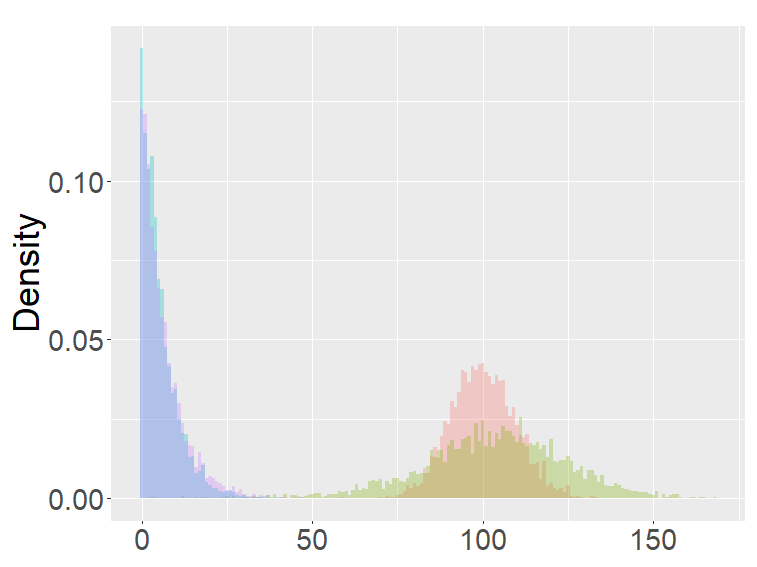}

\includegraphics[width=.45\textwidth]{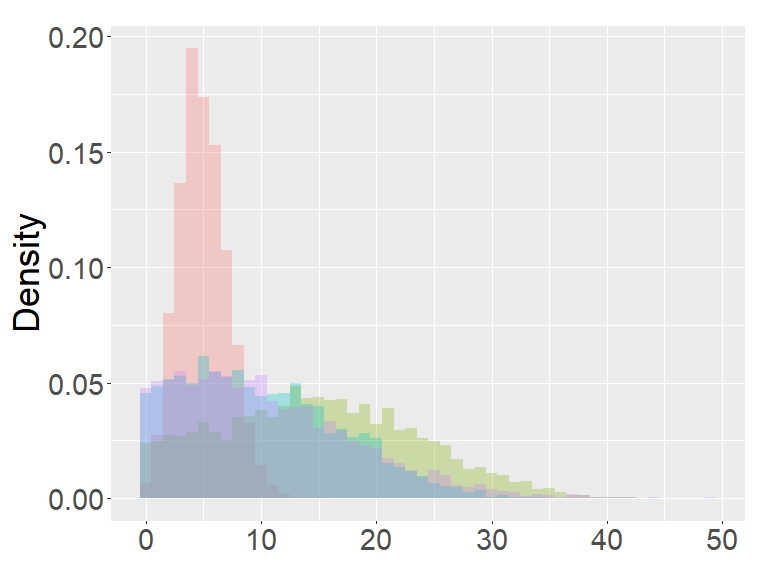}
\includegraphics[width=.45\textwidth]{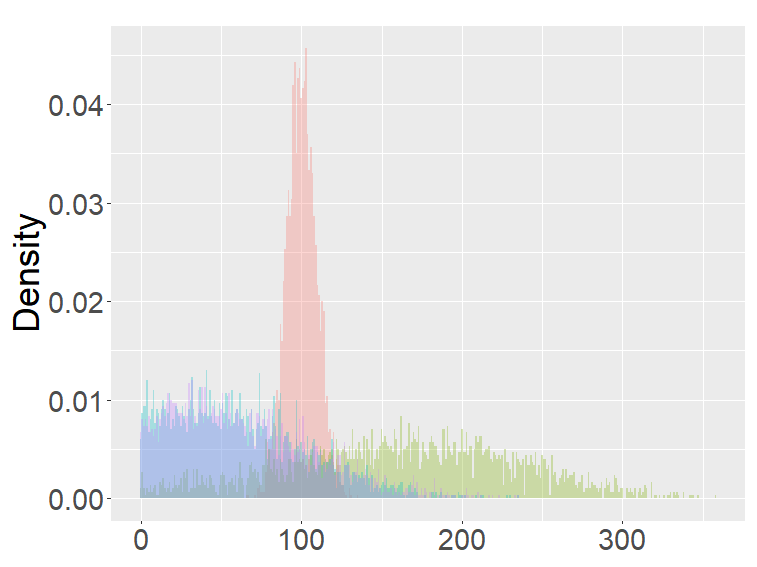}

\includegraphics[width=.45\textwidth]{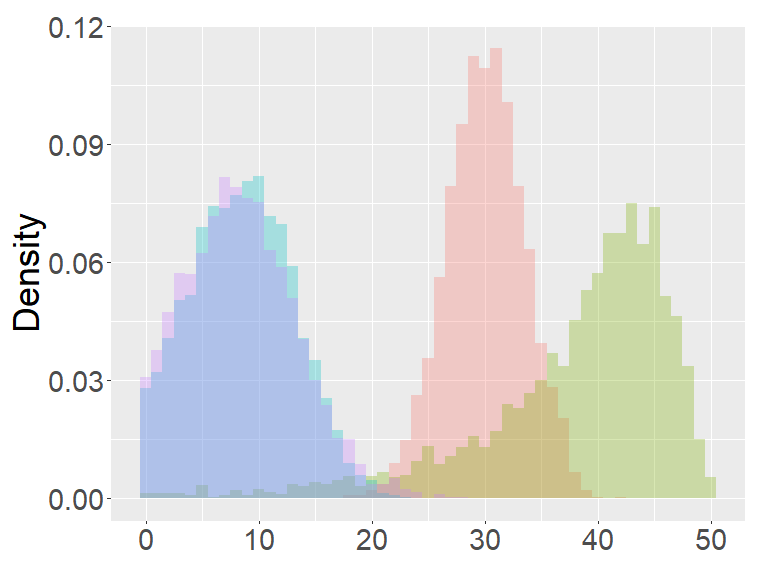}
\includegraphics[width=.45\textwidth]{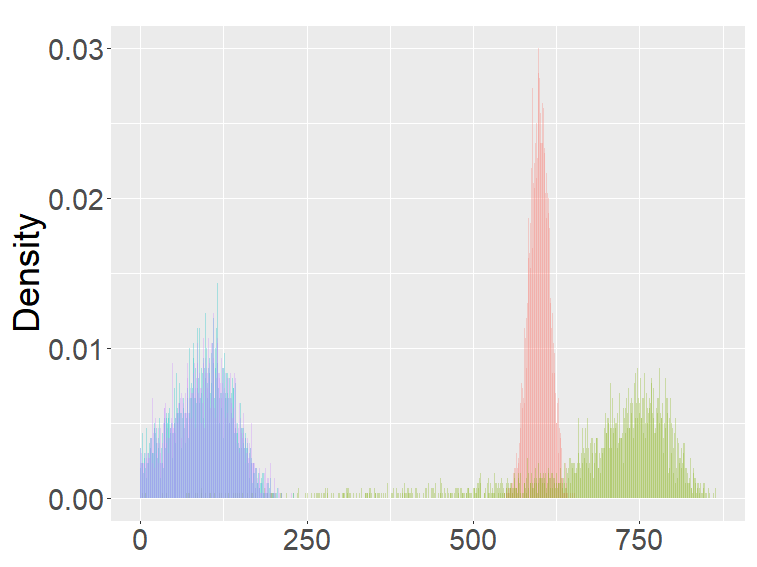} 
\includegraphics[width=1\textwidth]{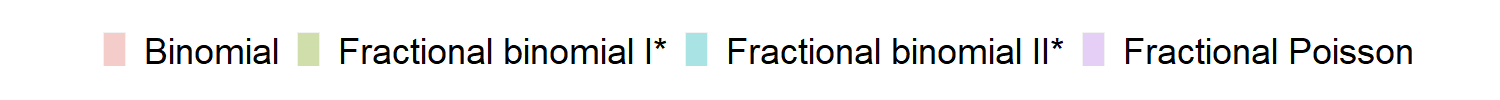} 
\caption{The first row: $(p,H,c)=(.1, .6, .2)$ with $n=50$ (left), 1000 (right),\\
the second row: $(p,H,c)=(.1, .8, .6)$ with $n=50$ (left), 1000 (right), \\
the third row: $(p,H,c)=(.6, .9, .3)$ with $n=50$ (left), 1000 (right).}
\label{Figure 4}
\end{figure}
\newpage

\section{Applications}
\subsection{ Application of GBP}
The monthly unemployment rate (seasonally adjusted, percent) from  January 1948 to December 2022 was obtained from the U.S. Bureau of Labor Statistics (\url{https://fred.stlouisfed.org/series/UNRATE}) and shown in the left graph of Figure 5. Using the decomposition method,  the trend of the time series of the unemployment rate was extracted, and the remaining component of the detrended unemployment rate was used for data analysis. With the detrended time series, a sequence of indicator variables was made to specify if the detrended unemployment rate is above a cutoff (indicator 1) or not (indicator 0) at each time. We tried three cutoffs, .08\%, .21\%, .28\%, and three indicator sequences were obtained. In Figure 5, the graph on the right shows the detrended unemployment rates and horizontal lines at the three cutoffs. 
\begin{figure}[htp]
    \centering
\includegraphics[width=.39\textwidth]{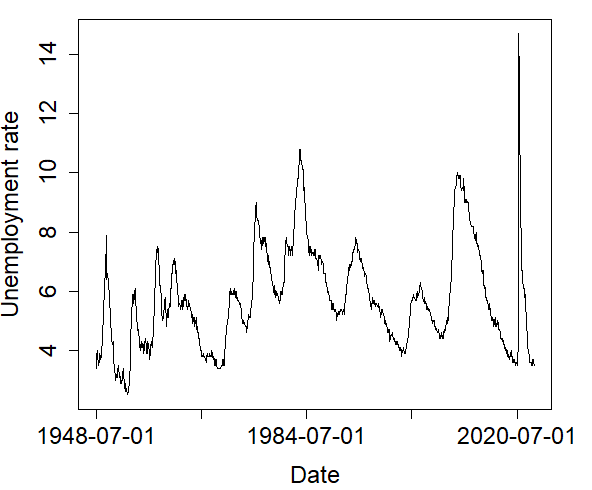}
\includegraphics[width=.39\textwidth]{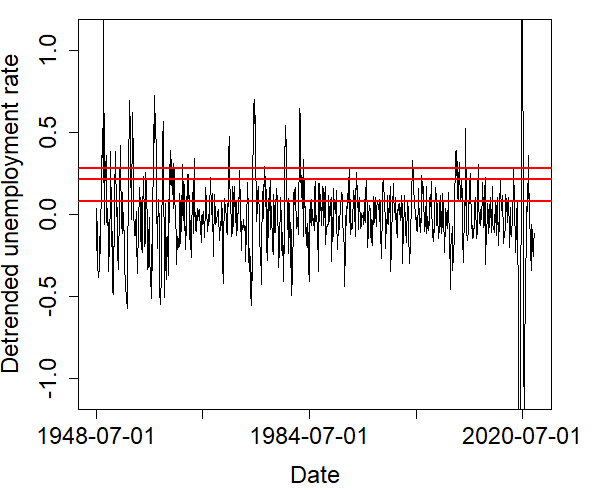} 
 \caption{Left: monthly unemployment rate (percent), Right:  detrended unemployment rate with horizontal line at .08\%, .21\%, .28\%.  }
 \end{figure} 
For each indicator sequence, we checked its autocorrelation in the left graphs of Figure 6.  It is observed that in the indicator sequence with cutoff .08, the correlation decreases relatively fast, and as the cutoff increases, the decay rate of the correlation slows down. 
 \begin{figure}  
\centering
\includegraphics[width=.35\textwidth]{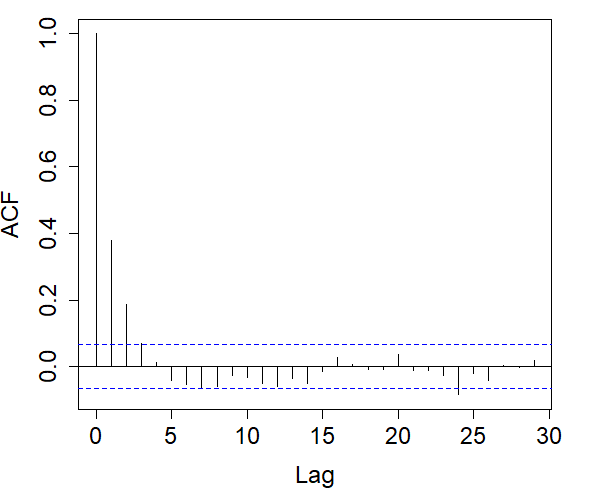} 
\includegraphics[width=.35\textwidth]{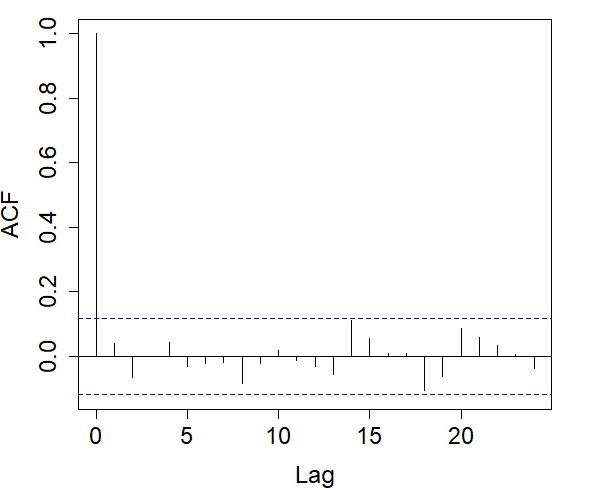} \vspace{.1cm}\\ 
 \includegraphics[width=.35\textwidth]{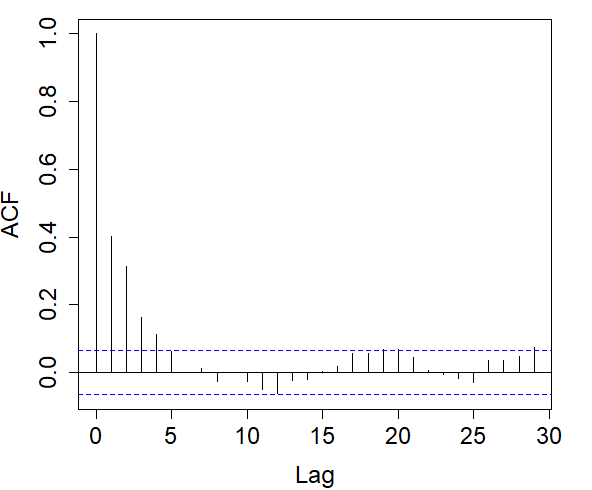}
 \includegraphics[width=.35\textwidth]{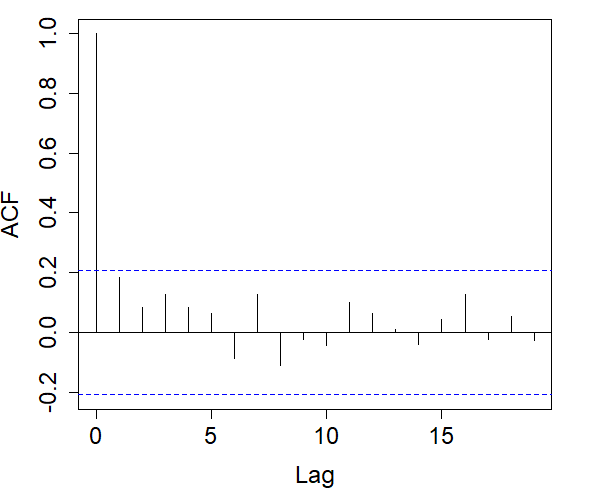} \vspace{.1cm}\\
 \includegraphics[width=.35\textwidth]{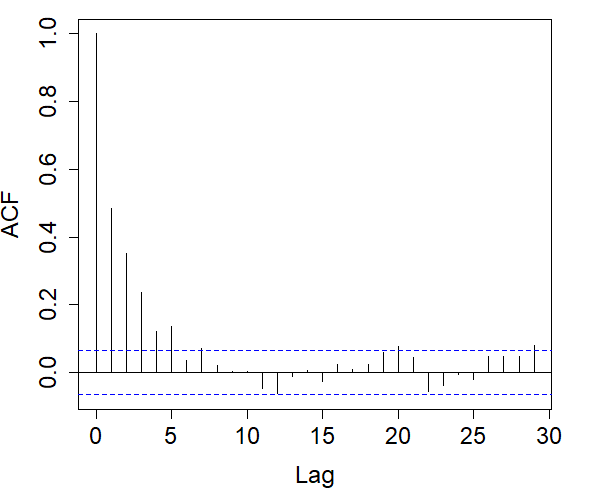}
 \includegraphics[width=.35\textwidth]{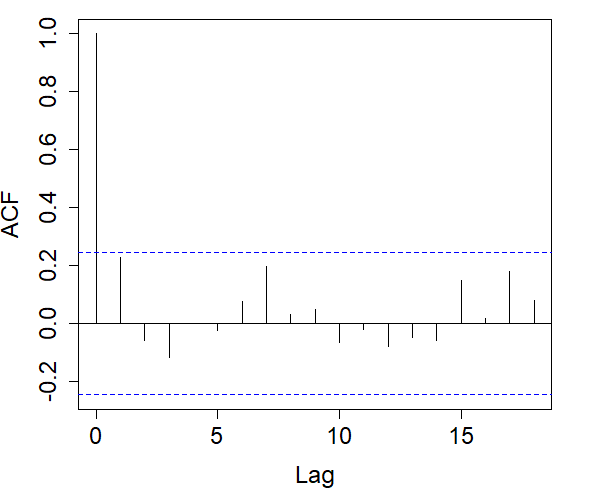}
    \caption{Left: Autocorrelograms of the indicator sequence of detrended unemployment rate with cutoff .08\% (top), .21\% (middle), .28\% (bottom).\\
    Right: Autocorrelograms of return times in the indicator sequence with cutoff .08\% (top), .21\% (middle), .28\% (bottom). }
    \label{fig:my_label}
\end{figure}

The GBP-I, GBP-II, Bernoulli process, and Markov chains were applied to each indicator sequence. In each of the four stochastic models, the return times, times between successive 1's, are independent and identically distributed as
\begin{align}
& P(\text{return time}=k) \nonumber \\&=\begin{dcases}
L_H(\{1,2\})I_{\{k=1\}}+I_{\{k>1\}}D_H(\{1, k+1\},\{2,3,\cdots, k\} ) &\text{ in  GBP-I,} \\cL_H^{\circ}(\{1,2\})I_{\{k=1\}} +I_{\{k>1\}}cD_H^{\circ}(\{1, k+1\},\{2,3,\cdots, k\} ) 
&\text{ in  GBP-II,}\\ (1-p)^{k-1} p &\text{ in Bernoulli process,} \\
pI_{\{k=1\}} + I_{\{k>1\}}(1-p) (1-q)^{k-2}q &\text{ in Markov chain,}
\end{dcases}     
\end{align}
where $k\in \mathbb{N}$, 
 $p, q \in (0,1)$, and $I_{\{k=0\}}, I_{\{k>0\}}$ are indicator variables.

 In each indicator sequence, we checked whether the return times were correlated. Autocorrelograms on the right in Figure 6 show that at a 5\% significance level, there is no significant evidence that return times are correlated. Therefore, we could proceed to fit each of the distributions in (7) to the return times of each indicator sequence.

   The results on the MLE of the parameters and the AIC of each model are shown in Tables 1-3. From Table 1, when the cutoff is .08, the Markov chain shows the best fit, having the smallest AIC, followed by the GBP-I. 
Table 2 shows that with the cutoff of .21, GBP-I has the smallest AIC, although the difference from the second smallest AIC from the Markov chain is very small. 
In Table 3, when the cutoff is .28, the GBP-I best fits the data, followed by the GBP-II and the Markov chain, in that order. The graphs in Figure 7 show the histogram of the return time overlaid with the fitted probability models.

\begin{table}[htp] \centering
\begin{tabular}{ |p{3cm}||p{2cm}|p{3cm}|p{2cm}| }
 \hline
 Probability model& Parameters&MLE& AIC
\\ \hline GBP-I  & ($p,H,c$)& (.30, .11, .23) & 974.68\\
GBP-II &($H,c$) &  (.85, .56) & 1027.04\\
Bernoulli process & $p$&.31 &  1079.06 \\
Markov chain & ($p,q$) &(.57, .19)&   961.55 \\
 \hline
\end{tabular}
\caption{Indicator sequence with the cutoff  .08}

\begin{tabular}{ |p{3cm}||p{2cm}|p{3cm}|p{2cm}| }
 \hline
 Probability model& Parameters&MLE& AIC\\
 \hline
 GBP-I  & ($p,H,c$)& (.09, .46, .33) & 479.76\\
GBP-II &($H,c$) &  (.77, .45) & 496.18\\
Bernoulli process & $p$&.10&  570.96 \\
Markov chain & ($p,q$) &(.47, .06)&  480.38 \\
 \hline
\end{tabular}
\caption{Indicator sequence with the cutoff  .21}

\begin{tabular}{ |p{3cm}||p{2cm}|p{3cm}|p{2cm}| }
 
 \hline
 Probability model& Parameters&MLE& AIC\\
 \hline
 GBP-I  & ($p,H,c$)& (.06, .58, .42) & 340.56\\
GBP-II &($H,c$) &  (.75, .53) & 346\\
Bernoulli process & $p$ &.07 &   458.44  \\
Markov chain & ($p,q$) &(.53, .04)& 348.41 \\
 \hline
\end{tabular}
\caption{Indicator sequence with the cutoff .28}
\label{table:4}
\end{table}

\begin{figure}
    \centering

\includegraphics[width=.45\textwidth]{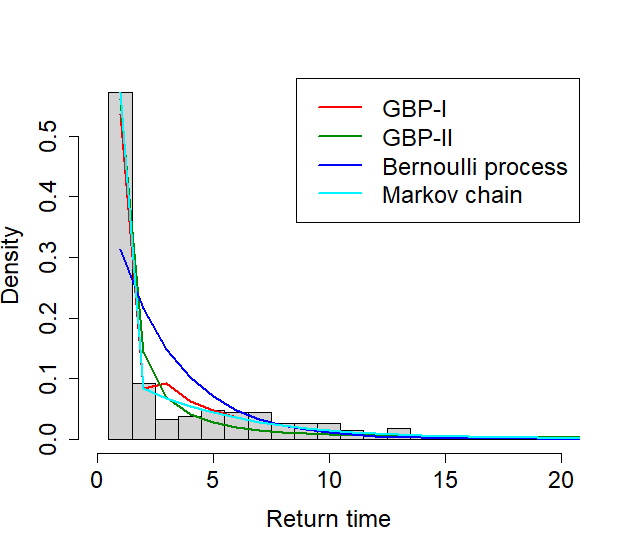}
\includegraphics[width=.45\textwidth]{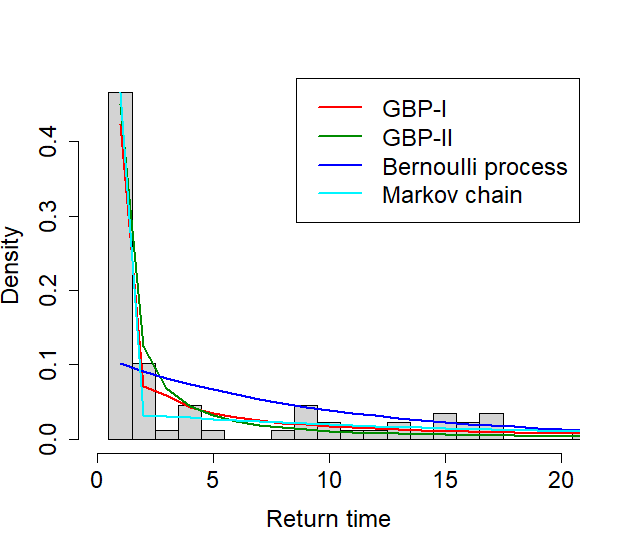}
\includegraphics[width=.45\textwidth]{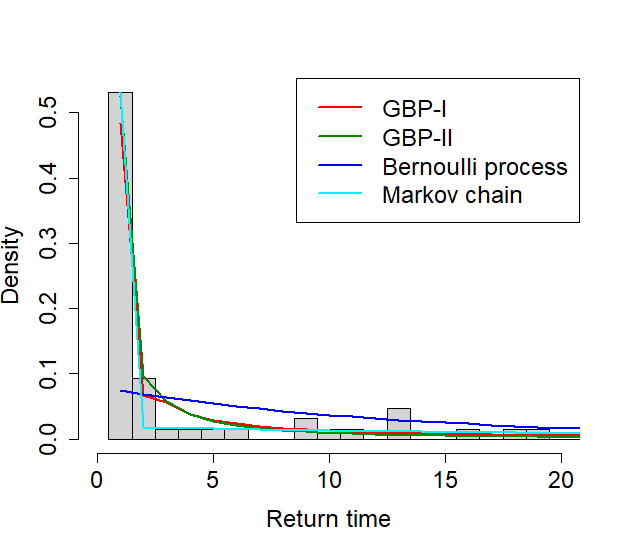}

    \caption{Data and fitted  distributions of the return time in the indicator sequence with cutoff .08 (top left), .21 (top right), .28 (bottom)}
    \label{fig:my_label}
\end{figure}

 From the results, it is observed that the relative performance of the GBP-I to the Markov chain improves as the cutoff increases.    It seems to be related to the fact that as the cutoff increases, the estimated parameter of $H$ in the GBP-I also increases. Especially, with the largest cutoff, the estimate of $H$  was $.58$, which indicates the presence of long-range dependence (LRD) in the indicator sequence. Since GBPs can incorporate LRD in a binary sequence unlike a Markov chain, it is not surprising that GBP shows a better fit than a Markov chain in the presence of LRD.

\subsection{Application of fractional binomial distribution}
We use a dataset in horticulture in \cite{Ridout1998ModelsFC}.  The dataset contains the number of roots produced by 270 micropropagated shoots of the columnar apple cultivar Trajans which were cultured under different experimental conditions, and the distribution of the number of roots is over-dispersed and has excess zeros.
In \cite{Ridout1998ModelsFC}, it was shown that zero-inflated Poisson (ZIP) and zero-inflated negative binomial (ZINB) models fitted the data better than Poisson and negative binomial models with covariates on the experimental conditions. 
Here we only use data on the number of roots without covariates, and fit fractional binomial models, ZIP, and ZINB.
The results on the MLE of parameters and AIC of each model are provided in Table 4.  The FB-I shows the lowest AIC, followed by the FB-II, zero-inflated models, FB-II$^*$, and FB-I$^*$, in that order.  Figure 8 shows the fitted distribution of each model with the data distribution. It was observed that FB-I and FB-II fitted the data well over the entire range of distribution, whereas the zero-inflated models overestimated the variable around the middle of the range and underestimated it in the latter half of the range. FB-I$^*$, II$^*$ show the worst fit to the data distribution, and also since the estimated $p$ is almost zero in FB-I$^*$, the fitted distribution of FB-I$^*$ and FB-II$^*$ becomes almost identical.

\begin{table}
\begin{tabular}{ |p{3cm}||p{2cm}|p{3cm}| p{2cm}|}
 
 \hline
 Probability model& Parameters&MLE&AIC\\
 \hline
 FB-I & ($p,H,c$)& (.30, .74 ,.31)  &  1348.44\\
 FB-II &($\lambda,H,c$) &  (.47, .92, .57)&   1350.1\\
 ZINB& $(\lambda,r, \pi) $ &(6.59, 9.97, .23)& 1358.6  \\
 ZIP& ($\lambda, \pi$) & (6.62, .24)&  1381.6  \\
  FB-II$^*$  & ($H,c$) &(.77, .70) & 1420.24\\
 FB-I$^*$ & ($p,H,c$)&(.00, .77, .70) & 1422.24  \\
 \hline
\end{tabular}
\caption{MLE and AIC for roots dataset}
\label{table:4}
\end{table}

\begin{figure}
    \centering
    \includegraphics[width=.5\linewidth]{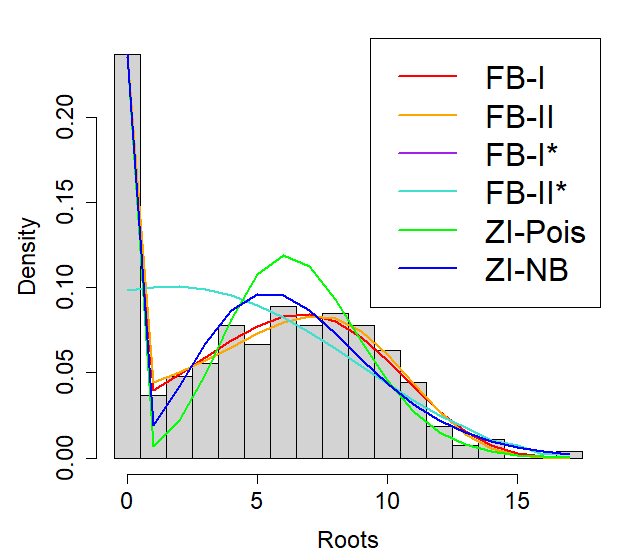}
    \caption{Fitted and data distribution on the number of roots}
    \label{fig:enter-label}
\end{figure}

\section{Conclusion}
We proposed generalized Bernoulli processes that are stationary binary sequences and found a connection to the fractional Poisson process. 
GBPs can possess long-range dependence, and the interarrival time of GBPs follows a heavy-tailed distribution.  Since a GBP can have the same scaling limit as the fractional Poisson process, it can be considered as a discrete-time analog of the fractional Poisson process.  
Fractional binomial distributions are defined from the sum in GBPs and possess various shapes, from highly skewed to flat.

 GBPs were applied to economic data with indicator variables, and the model fit of the GBPs was compared to the model fit of a Markov chain. It turned out that in the presence of LRD, the GBPs outperformed the Markov chain, which can be explained by the fact that the GBPs can incorporate LRD. It was noted that LRD appeared in the dataset when the higher cutoff was used for the indicator sequence, which suggests a  connection between rare events and LRD, and this shows potential for the applicability of GBPs for modeling LRD in rare events.

Fractional binomial models were applied to count data with excess zeros.
 It was shown that a fractional binomial model fitted the data better than zero-inflated models that are extensively used for overdispersed, excess zero count data. 

%
%

%

\begin{supplement}
\stitle{Generalized Bernoulli process and fractional Poisson process: supplemental document}
\sdescription{All the proofs of the proposition and theorems of this article are included in the supplementary material.}
\end{supplement} 

\begin{supplement}
\stitle{Computer code for simulations and data applications}
\sdescription{Code for simulations can be found in  J. Lee, frbinom, (2023), GitHub repository, https://github.com/leejeo25/frbinom.  Computer code used in the application of Section 5 is provided in 
J. Lee, GBP\_FB, (2023), GitHub repository, 
https://github.com/leejeo25/GBP\_FB}
\end{supplement} 


\bibliography{_11_16.bib}

\end{document}